\documentclass[12pt]{amsart}
\usepackage[utf8]{inputenc}
\usepackage[T1]{fontenc}
\usepackage{amsmath,amsfonts,amsthm,amssymb}
\usepackage{enumitem}
\usepackage{graphicx}
\usepackage{pgf,tikz,pgfplots}
\usetikzlibrary{arrows}
\usetikzlibrary{angles,quotes}
\usetikzlibrary{calc}
\usepgfplotslibrary{fillbetween}
\usetikzlibrary{positioning}
\usetikzlibrary{calc,shapes.geometric}
\usetikzlibrary{decorations.pathreplacing}
\usetikzlibrary{cd}
\usetikzlibrary{positioning}
\usepackage{mathrsfs}
\usepackage{caption,subcaption}
\usepackage{mathtools}
\usepackage{xspace}
\usetikzlibrary{patterns}
\calclayout
\usepackage{soul}
\usepackage{csquotes}
\usepackage{float}
\usepackage{comment}
\usepackage{url}
\usepackage{hyperref}
\hypersetup{
  colorlinks   = true,          
  urlcolor     = blue,          
  linkcolor    = green,          
  citecolor   = orange             
}

\newtheorem{theorem}{Theorem}
\newtheorem{corollary}[theorem]{Corollary}
\newtheorem{conjecture}[theorem]{Conjecture}

\newtheorem{lemma}[theorem]{Lemma}

\theoremstyle{definition}
\newtheorem{definition}[theorem]{Definition}
\newtheorem{example}[theorem]{Example}
\newtheorem{problem}[theorem]{Problem}

\theoremstyle{remark}
\newtheorem{remark}[theorem]{Remark}

\newcommand{\M}{\mathsf{M}}
\newcommand{\N}{\mathsf{N}}
\newcommand{\B}{\mathsf{B}}

\newcommand\doi[1]{\href{http://dx.doi.org/#1}{\texttt{doi:#1}}}

\graphicspath{{Figures/}}

\tikzstyle{lattice} = [draw=red, fill=red]
\tikzstyle{valid} = [draw=red, thick, fill=white]
\tikzstyle{intersect} = [draw=orange, fill=orange]
\tikzstyle{boundary} = [draw=blue, fill=blue]
\tikzstyle{triangle} = [draw=black, thick, fill=blue!20]
\tikzstyle{inequality} = [draw=green, thick]
\tikzstyle{someline} = [draw=black, dashed]

\title[Transversal matroids and the half plane property]{Transversal matroids and the half plane property}
\author{Ayush Kumar Tewari}

\address[A.~K.~Tewari]{
}
\email{tropical.tewari@gmail.com}

\makeatletter
\@namedef{subjclassname@1991}{Mathematics Subject Classification (2020)}
\makeatother



\thanks{I am extremely grateful to 
Adrien Kassel for discussing and explaining his work on determinantal probability measures on matroids. I am also grateful to Joseph Bonin for going through an earlier draft of this work and for also discussing his work on excluded minor characterisation of positroids. I am also thankful to Mario Kummer for providing valuable feedback on this work. }

\usepackage{pgfplots}
\pgfplotsset{compat=1.18}

\begin{document}

\begin{abstract}
We focus on checking the validity of the half-plane property on two prominent classes of transversal matroids, namely lattice path matroids and bicircular matroids. We show that lattice path matroids satisfy the half-plane property. Subsequently, we show an explicit example of a bicircular matroid that is not a positroid  and discuss the negative correlation properties of bases of transversal matroids. We prove that sparse paving matroids do not satisfy the Rayleigh property, which helps us gain new perspectives about conjectures on negative correlation in basis elements of matroids in general.
\end{abstract}

\maketitle

{
\hypersetup{linkcolor=blue}
\tableofcontents
}

\section{Introduction}

The class of transversal matroids introduced in \cite{edmonds1965transversals} is one of the most prominent classes of matroids. Since then several subclasses of transversal matroids have been studied in detail owing to their various applications. In this article, our main focus is on two prominent subclasses of transversal matroids: lattice path matroids \cite{bonin2003lattice} and bicircular matroids \cite{pereira1972subgraphs}. The central theme of this article is to study the geometry of the basis-generating polynomials of these two subclasses of transversal matroids. We discuss the broader context in which this brings in concepts from the geometry of polynomials and with recent work on Lorentzian polynomials \cite{branden2020lorentzian}, this seems intricately related to various other classes of matroids. 

One of the first studies regarding the geometry of basis-generating polynomial of matroids with inspirations from combinatorics associated with electrical networks was carried out in \cite{choe2004homogeneous} where the authors established various results concerning the \emph{half-plane property}\footnote{In literature, this property is also referred to as the strong half-plane property to distinguish with the weak half-plane property.} and they provide examples of classes of matroids that satisfy this property, which includes regular matroids, uniform matroids and the sixth-root of unity matroids. They also introduce the definition of \emph{nice} transversal matroids, and show that nice transversal matroids satisfy the half-plane property. The interest in the class of transversal matroids stems from the fact that they do satisfy a mild version of the half-plane property, namely the \emph{weak half-plane property} \cite[Corollary 10.3]{choe2004homogeneous} and the authors envisaged that it might be the case that all transversal matroids might satisfy the half-plane property \cite[Question 13.17]{choe2004homogeneous}. However, with subsequent works \cite{choe2006rayleigh, branden2007polynomials}, we have the hindsight to say that this is not the case, although the question remains open for subclasses of transversal matroids. Also, in \cite{branden2007polynomials, borcea2009negative} the authors highlight the natural way in which the half-plane property is linked with other geometric properties of the basis-generating polynomial for matroids, namely the \emph{Rayleigh} and \emph{strongly Rayleigh} properties. In this direction, one of our first results shows that the class of lattice path matroids satisfies the half-plane property.

\begin{theorem}\label{thm:thm1}
Lattice path matroids satisfy the half-plane property.    
\end{theorem}

Another aspect of our work in this article is to show how the classes of bicircular matroids and positroids interact in the superclass of gammoids. We know that lattice path matroids are positroids \cite{oh2011positroids} and with the recent work on \emph{lattice path bicircular matroids} \cite{guzman2022lattice} we know of the existence of a subclass of bicircular matroids that are positroids. We firstly note with the help of an example that the classes of bicircular matroids and positroids do not lie inside each other and this is stated in Corollary \ref{cor:bicirc_and_pos}. The proof rests on the recent work on the classification of the linear order on the ground set of a positroid, also referred to as the \emph{positroid order} in \cite{bonin2023characterization} and we use an excluded minor characterization of positroids to illustrate an example of a rank four bicircular matroid which is not a positroid (Example \ref{eg:bicicrc_not_pos}), which in turn comes from the definition of \emph{base sortable matroids} \cite{blum2001base}, which are now known to be equivalent to positroids \cite{bonin2023characterization}.  Our findings concerning the various classes of matroids are summed up in Figure \ref{fig:venn_diag_matroid}.



Lastly, we consider the class of sparse paving matroids, which is known to dominate in the enumeration of matroids on a log scale and is conjectured to be the dominating class of matroids in the enumeration of matroids in general. By invoking results based on $2-$ sums of Rayleigh matroids and properties of 3-connected matroids, we are able to prove the following result concerning sparse paving matroids.

\begin{corollary}\label{cor:cor1}
The class of sparse paving matroids is not Rayleigh.    
\end{corollary}

Our findings in the form of Corollary \ref{cor:cor1} provide new perspectives for verifying Conjecture \ref{conj:lorent_conj} from \cite{branden2020lorentzian} for the class of matroids, and this result suggests that Rayleigh matroids might not necessarily dominate in the class of all matroids, so it might be worthwhile to search for classes of matroids which are $c-$\emph{Rayleigh}, for $c > 1$, since till now we only know of sporadic examples of such matroids.

\section{Preliminaries}

We refer the reader to \cite{oxley} for basics on matroid theory. A \emph{matroid} of rank $k$ on the ground set $E(\M)$ is a nonempty collection $\M \subseteq \binom{E}{k}$ of $k$-element subsets of $E(\M)$, called \emph{bases} of $\M$ (also referred as the set of independent elements), that satisfies the basis exchange axiom:
\begin{center}
For any $I , J \in \M$ and $a \in I$, there exists $b \in J$ such that $I \setminus \{ a \} \cup \{ b \} \in \M$    
\end{center}

A matroid is called \emph{representable} or \emph{matric} if it can be represented by columns of a matrix over some field $\mathbb{K}$. We index the columns of a $k \times n$ matrix by the set $[n]$. A \emph{positroid} $\mathsf{P}$ of rank $k$ is a  matroid $\M$ that can be represented over $\mathbb{R}$ by a $k \times n$-matrix $A$ such that the maximal minor $p_I$ is non-negative for each $I \in \binom{[n]}{k}$. An important fact about the class of positroids is that it is closed under taking minors and duality.

We also take this opportunity to define the operation of \emph{2-sums} for matroids \cite{oxley}, which appears in our work.

\begin{definition}
Let $\M$ and $\N$ be matroids, each with at least two elements. Let $E(\M) \cap E(\N) = \{p\}$ and suppose that neither $\M$ nor $\N$ has $\{p\}$ as a \emph{separator}. Then the $2-$sum $\M \oplus_{2}  \N$ of $\M$ and $\N$ is $S(\M, \N) / p$ or $P(\M, \N)\backslash p$, where $S(\M, \N)$ and $P(\M, \N)$ represent \emph{series connecion} and \emph{parallel connection} of $\M$ and $\N$ respectively. The element $p$ is called the \emph{basepoint} of the $2-$sum, and $\M$ and $\N$ are called the parts of the $2-$sum.    
\end{definition}

If $\mathcal{C}(\M)$ represents the set of circuits of the matroid $\M$, then the description of the set of circuits of the 2-sum $\M \oplus_{2} \N$ is given as  \cite[Proposition 7.1.20]{oxley},

\[ \mathcal{C}(\M \backslash p) \cup \mathcal{C}(\N \backslash p) \cup \{ (C \cup D) - p :  p \in C \in \mathcal{C}(\M) \> \text{and} \> p \in D \in \mathcal{C}(\N)   \}  \]

There are many different families of matroids, where the elements in the bases correspond to various properties of the underlying structure of the set. One prominent example of such matroids is \emph{cyle matroids} where for an undirected graph $G=(V,E)$, the underlying ground set is the set of edges $E$ and the bases elements of the graphical matroid $\M(G)$ are the spanning forests of $G$. A matroid that is isomorphic to the cycle matroid of a graph is called \emph{graphic}. The class of graphic matroids is minor-closed and is a subclass of \emph{regular} matroids.

Let $E$ be a set (which is going to be the ground set of the matroid), A \emph{set system} $(S, A)$ is a set $S$ along with a multiset $\mathcal{A} = (A_{j} : j \in J)$ of subsets of $S$. A \emph{transversal} of $\mathcal{A}$ is a set $\{ x_{j} : j \in J  \}$ of  $|J|$ distinct elements such that $x_{j} \in  A_{j}$ for all $j \in J$. A \emph{partial transversal} of $\mathcal{A}$ is a transversal of a set system of the form ($A_{k} : k \in K) $ with $K$ a subset of $J$. A \emph{transversal matroid} is a matroid whose independent sets are the partial transversals of some set system $\mathcal{A} = (A_{j} : j \in J)$ and $\mathcal{A}$ is called the \emph{presentation} of the transversal matroid. We denote this matroid by $\M[\mathcal{A}]$. The bases of a transversal matroid are the maximal partial transversals of $\mathcal{A}$ \cite{bonin2010introduction,bonin2003lattice}. The class of transversal matroids is neither closed under taking minors nor under duality. However, if we consider the class of \emph{gammoids} introduced by Hazel \cite{nuij1968note}, then we know that gammoids are the smallest class of matroids that includes the transversal matroids and is closed under duality and taking minors.

Let $G = (V, E)$ be a graph and for $I \subseteq E$ define $G[I]$ to be the edge-induced subgraph of $G$. The collection
\begin{center}
    $\{ I \subseteq E \> | \> \text{each \> component \> of \> } G[I] \> \> \text{has \> atmost \> one \> cycle} \}$
\end{center}

is the collection of independent sets of a matroid on $E$. This matroid is called the \emph{bicircular matroid} of $G$ and is denoted by $\B(G)$. A matroid is bicircular if there exists a graph $G$ such that $\M=\B(G)$. The graph $G$ is a \emph{bicircular representation} of $\M$ \cite{coullard1991representations}.

It is well known that bicircular matroids are transversal and if $\B(G)$ is a bicircular matroid for the graph $G = (V,E)$, then for each $v\in V$, the family of sets

\[ \mathcal{A}_{v} = \{ e \in E \> | \> e \> \text{is incident  with} \> v \} \]

provides the presentation of $\B(G)$ as a transversal matroid \cite[Theorem 3.1]{matthews}. Additionally, bicircular matroids are precisely those transversal matroids that have a presentation such that the intersection of any three
members is empty \cite[Corollary 3.3]{matthews}.

Equivalently, the bicircular matroid $\B(G)$ of $G=(V,E)$, is the matroid on $E$ whose circuits, which are called the \emph{bicycles} of $G$, are the edge set of a subgraph of $G$ which is a subdivision of one of the graphs shown in Figure \ref{fig:bicirc_def} \cite{oxley, coullard1991representations}. The class of bicircular matroids is closed under taking minors but is not closed under duality. 

\begin{figure}
    \centering
    \includegraphics[scale=0.5]{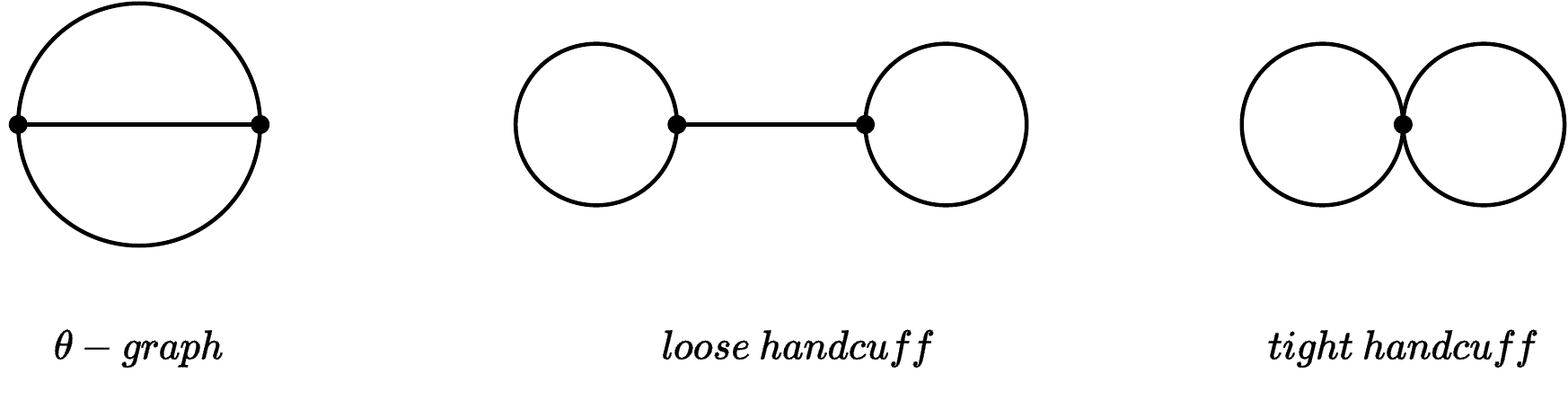}
    \caption{$\theta-$graph, loose handcuﬀ and tight
handcuﬀ graphs.}
    \label{fig:bicirc_def}
\end{figure}

Another prominent class of transversal matroids that show up in our discussion is the class of \emph{lattice path matroids}.

\begin{definition}[Definition 3.1 \cite{bonin2003lattice}]
Let $P = p_{1}, p_{2} \hdots p_{m+r}$ and $Q = q_{1}, q_{2},  \hdots q_{m+r}$ be two lattice paths from
$(0,0)$ to $(m,r)$ with $P$ never going above $Q$. Let $\{p_{u_{1}}, \hdots  p_{u_{r}}\}$  be the set of North steps
of $P$; with $u_{1} < u_{2}  < u_{r}$; similarly, let $\{q_{l_{1}}, \hdots  q_{l_{r}}\}$ be the set of North steps of $Q$;
with $l_{1} < u_{2}  < l_{r}$: Let $N_{i}$ be the interval $[l_{i}, u_{i}]$ of integers. Let $\M[P, Q]$ be the transversal matroid that has ground set $[m + r]$ and presentation $(N_{i}: i \in [r])$; the pair $(P, Q)$ is a \emph{lattice path presentation} of $M[P, Q]$. A lattice path matroid is a matroid $\M$ that is isomorphic to $\M[P, Q]$ for some such pair of lattice paths $P$ and $Q$.    
\end{definition}

An important class of lattice path matroids are the \emph{snakes} or the \emph{border strip matroids} which are defined as follows \cite{knauer2018tutte},

\begin{definition}
We call a lattice path matroid $\M[P,Q]$ a \emph{snake} if it has at least two elements, it is connected and the strip contained between the paths $P$ and $Q$ does not contain any interior lattice point.   
\end{definition}

The class of lattice path matroids is closed both under duality and taking minors. Also, by the work of Oh \cite{oh2011positroids} we know that lattice path matroids are also positroids. Lattice path positroids have been generalized to multi-path matroids \cite{bonin2007multi} which again are transversal matroids, closed under taking minors and duals and are also positroids \cite{bonin2023characterization}. 

\begin{figure}
    \centering
    \includegraphics[scale=0.5]{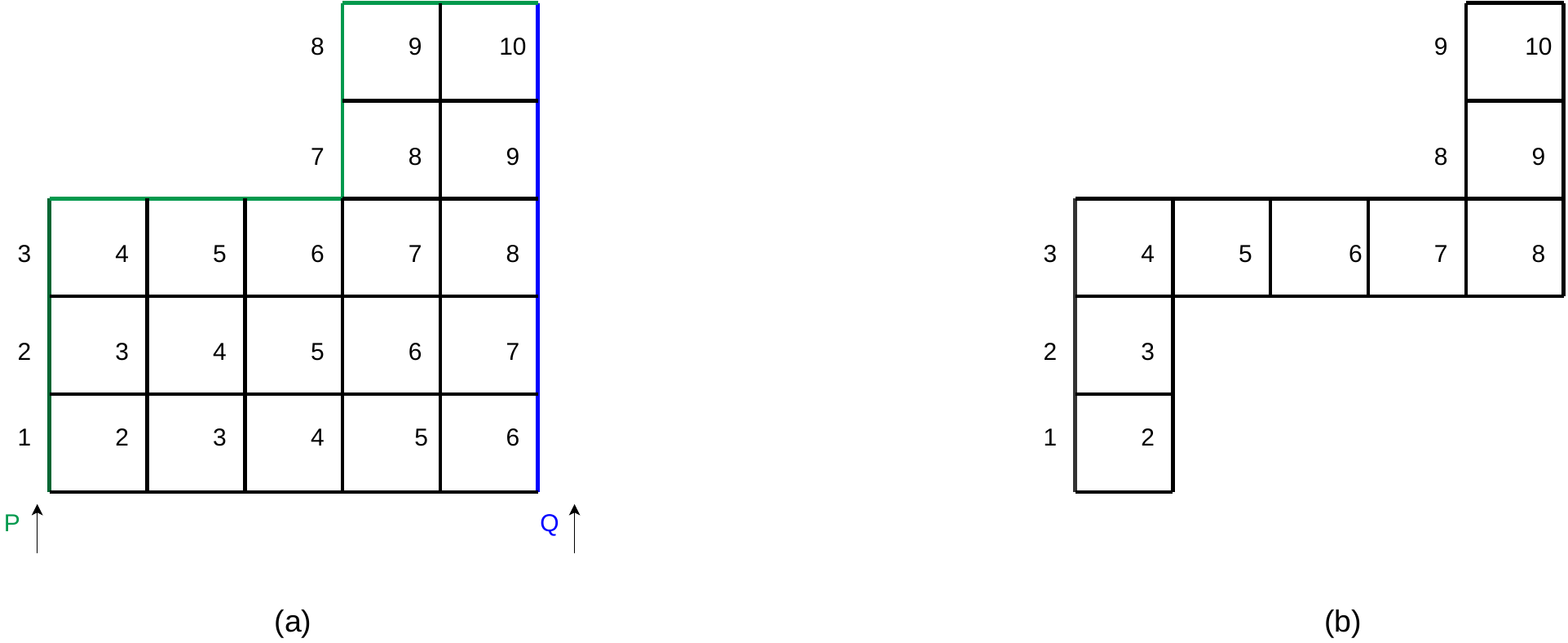}
    \caption{An example of a lattice path matroid (a) $\M = \M[12378,678910]$ and (b) represents an example of a snake.}
    \label{fig:enter-label}
\end{figure}

\section{Negative Dependence and Matroids}

We recall most of our definitions from \cite{borcea2009negative}.
Let $\mathfrak{P}^{n}, n \in \mathbb{N}$ be the set of all probability measures on the Boolean algebra $2^{[n]}$. Let $\mathcal{P}_{n}$ be the set of all multi-aﬃne polynomials in $n$ variables $f(z_{1}, \hdots , z_{n})$ with non-negative coeﬃcients such that $f(\mathbf{1}) = 1,$ where $\mathbf{1} = (1, \hdots , 1) \in \mathbb{R}^{n}$ denotes the
all ones vector. There exists a one-to-one correspondence between elements in $\mathfrak{P}^{n}$ and $\mathcal{P}_{n}$ \cite{borcea2009negative}, which is described as follows: Consider $\mu \in \mathfrak{P}^{n}$, then the generating polynomial for $\mu$ can be written as 

\[ g_{\mu}(\mathbf{z}) = \int \mathbf{z}^{s}d\mu(S) = \sum_{S\subseteq [n]} \mu(S) \mathbf{z}^{s} , \quad \mathbf{z}  = (z_{1}, \hdots,z_{n}), \quad \mathbf{z}^{s} = \prod_{i \in S} z^{i} \]

and if

\[  f(\mathbf{z}) = \sum_{S \subseteq [n]} a_{S} \mathbf{z}^{S} \quad \in \mathcal{P}_{n} \]

then f defines a measure $\mu_{f}$ on $2^{[n]}$ by setting $\mu_{f}(S) = a_{S} , S \subseteq [n]$. It is clear that
$g_{\mu} \in \mathcal{P}_{n}, \mu_{f} \in \mathfrak{P}^{n}, g_{\mu_{f}} = f$, and $\mu_{g_{\nu}} = \nu$ for any $\nu \in \mathfrak{P}^{n}$, $f \in \mathcal{P}_{n}$. 

Let $\mu: 2^ {[n]} \rightarrow \mathbb{R} $ be a function that attains non-negative values on subsets $E \subseteq 2^{[n] }= \{1,2, \hdots, n\}$ satisfying $\sum_{E \subseteq 2^{[n]}} \mu(E) = 1$. Then $\mu$ is said to satisfy the \emph{negative lattice condition} (NLC) if

\[  \mu(S) \mu(T) \geq \mu(S\cup T) \mu(S \cap T)  \]

for all $S,T \subseteq [n]$. The corresponding probability measure defined on $2^{[n]}$ associated with $\mu$ is said to be \emph{negatively associated} if 

\[ \int Fd\mu \int Gd\mu \geq \int FGd\mu  \]




\begin{definition}
A polynomial $f \in \mathcal{P}_{n}$ is called a \emph{Rayleigh polynomial} if 

\[ \frac{ \partial f} {\partial z_{i}} (x) \frac{ \partial f} {\partial z_{j}} (x) \geq \frac{ \partial^{2} f} {\partial z_{i} \partial z_{j}} (x) f(x)  \] 

for all $x = (x_{1}, \hdots , x_{n}) \in \mathbb{R}^{n}_{+}$ and $1 \leq i, j \leq n$, where $\mathbb{R}_{+} = (0, \infty)$.
\end{definition}

More generally, a multi-aﬃne polynomial in $\mathbb{R}[z_{1}, \hdots, z_{n} ]$ with non-negative coeﬃcients is called a \emph{Rayleigh polynomial} if it satisﬁes the above condition. A measure $\mu \in \mathfrak{P}^{n}$ is said to be a \emph{Rayleigh measure} if its generating polynomial $g_{\mu}$ is Rayleigh \cite{borcea2009negative}.

The notion of \emph{Rayleigh} matroids was introduced in \cite{choe2006rayleigh}, wherein a matroid $\M$ is a \emph{Rayleigh} if its basis generating polynomial is a Rayleigh polynomial. A class of matroids is \emph{Rayleigh} or is said to satisfy the \emph{Rayleigh property} if the basis generating polynomials of all the matroids in it are \emph{Rayleigh}. A slightly stronger notion modeled on this property is the definition of \emph{strongly Rayleigh} polynomials,

\begin{definition}
A polynomial $f \in \mathcal{P}_{n}$ is called a \emph{strongly Rayleigh polynomial} if 

\[ \frac{ \partial f} {\partial z_{i}} (x) \frac{ \partial f} {\partial z_{j}} (x) \geq \frac{ \partial^{2} f} {\partial z_{i} \partial z_{j}} (x) f(x)  \] 

for all $x = (x_{1}, \hdots , x_{n}) \in \mathbb{R}^{n}$ and $1 \leq i, j \leq n$.    
\end{definition}

More generally, a multi-aﬃne polynomial in $\mathbb{R}[z_{1}, \hdots, z_{n} ]$ with non-negative coeﬃcients is called a \emph{ strongly Rayleigh polynomial} if it satisﬁes the above condition \cite{borcea2009negative}. It is clear from the definitions above that strongly Rayleigh condition implies the Rayleigh condition. Also, the notion of strongly Rayleigh is extended to matroids and classes of matroids in the same way as done for Rayleigh polynomials. 

We now recall the definition of a \emph{stable} polynomial, 

\begin{definition}
A polynomial $f \in \mathbb{C}[z_{1}, \hdots, z_{n}]$ is called \emph{stable} if $f(z_{1}, \hdots, z_{n}) \neq 0$ whenever $\text{Im} (z_{j}) > 0$ for $1 \leq j \leq n$. A stable polynomial with all real coeﬃcients is called \emph{real stable}.   
\end{definition}

The definition of stable polynomials coincides with the definition of polynomials which satisfy the \emph{half-plane property}, studied in detail in \cite{choe2004homogeneous}. This notion is extended to the class of matroids: a matroid $\M$ is said to satisfy the \emph{(strong) half-plane property} if the basis generating polynomial of $\M$ satisfies the half-plane property and such matroids are also referred as \emph{HPP matroids} \cite{choe2006rayleigh,choe2004homogeneous}. There is also a weaker notion of the half-plane property: a matroid $\M$ is said to satisfy the \emph{weak half-plane property} if there exists a stable polynomial whose support is the set of bases of the matroid $\M$. The matroid that satisfies the half-plane property is also strongly Rayleigh \cite{branden2007polynomials} and hence also Rayleigh.

A measure $\mu \in \mathfrak{P}^{n}$ is said to be a \emph{strongly Rayleigh measure} if its generating polynomial $g_{\mu}$ is real stable \cite{borcea2009negative}. In \cite[Theorem 5.6]{branden2007polynomials} it is shown that a multi-affine polynomial is stable if and only if it is strongly Rayleigh.

We also recall the notion of \emph{balanced matroids}, first introduced in \cite{feder1992balanced}, which uses the negative correlation property of elements in a matroid defined as follows

\begin{definition}
Let $\M$ be a matroid on the ground set $E$ with basis $\mathcal{B}(\M)$. Let $\mathbb{P}(e)$ denote the probability of an element $e$ being present in basis element $B$, which is chosen uniformly at random. The matroid $\M(E, \mathcal{B}(\M))$ is said to satisfy the  \emph{negatively correlated} property if 

\[  \mathbb{P}(ef) \leq \mathbb{P}(e)\mathbb{P}(f) \]

for all pair of distinct $e,f \in E$.
\end{definition}    

As is evident, the above definition is modeled on the definition of negatively associated measures defined previously.

\begin{definition}
A matroid $\M$ is said to be \emph{balanced} if all its minors including itself satisfy the negative correlation property.    
\end{definition}

It is known that Rayleigh matroids are balanced \cite{choe2006rayleigh}  although the converse is not true and this is described by an explicit example in \cite[Theorem 5.12]{choe2006rayleigh}.

We take this opportunity to touch upon other aspects related to the geometry of polynomials, which also have connections to linear programming,

We begin with some basic definitions,

\begin{definition}[Defintinion 2.8 \cite{kummer2023}]\label{def:determ}
A real homogeneous polynomial $h \in \mathbb{R}[x_{1} , \hdots , x_{n}]$ of degree $d$ is said to have a \emph{determinantal representation} if there are positive semi-definite matrices $A_{1} , \hdots, A_{n}$ of size $d \times d$ such that  

\[  h = \text{det}(x_{1}A_{1} + \hdots + x_{n}A_{n})  \]
\end{definition} 

$f$ is said to be \emph{weakly determinantal} if $f^{r}$ has a determinantal representation for some suitable $r \in \mathbb{N}$. A matroid is called \emph{weakly determinantal} if its basis-generating polynomial is weakly determinantal.

It is known by the work in \cite[Section 8]{choe2004homogeneous} that the bases generating polynomials of a matroid $\M$ has a determinantal representation if and only if $\M$ is a regular matroid. 

\begin{definition}
A multiaffine polynomial $h \in \mathbb{R}[x_{1}, \hdots , x_{n}]$ is called \emph{SOS-Rayleigh} if for all $1 \leq i, j \leq n$, the \emph{Rayleigh difference} $\Delta_{ij}(h)$,

\[  \Delta_{ij}(h)  = \frac{ \partial h} {\partial x_{i}} \frac{ \partial h} {\partial x_{j}} - \frac{ \partial^{2} h} {\partial x_{i} \partial x_{j}} \cdot h \]

is a sum of squares (of polynomials). A matroid is called \emph{SOS-Rayleigh} if its bases generating polynomial is SOS-Rayleigh.    
\end{definition}

It is clear from \cite[Theorem 5.6]{branden2007polynomials} that SOS Rayleigh implies the half-plane property. Additionally, weakly deteminantal also implies being SOS-Rayleigh \cite{kummer2023}. A matroid is said to be \emph{SOS-Rayleigh} if the basis generating polynomial of the matroid is SOS-Rayleigh.

We present some background on what motivates us to study the half-plane property in the context of transversal matroids. In \cite{choe2004homogeneous}, the authors explore whether the class of transversal matroids satisfies the half-plane property. They show that all transversal matroids satisfy the weak half-plane property. Subsequently, they also classify a class of transversal matroids that do satisfy the half-plane property and term them as \emph{nice} transversal matroids \cite[Corollary 10.3]{choe2004homogeneous}. With subsequent work we now know that not all transversal matroids satisfy the half-plane property: in \cite[Proposition 5.9]{choe2006rayleigh} and also mentioned with a correction in \cite{huh2021correlation}, the authors provide an explicit example of a rank four transversal matroid which is not Rayleigh and hence also does not satisfy the half-plane property. We discuss the background of this classification of nice transversal matroids.

We discuss the Heilmann-Leib Theorem \cite{heilmann1972theory} in the context of nice transversal matroids. Let $G = (V, E)$ be a loopless graph, then the matching polynomial with
edge weights $\{ \lambda_{e}\}_{e \in E}$ and vertex weights $\{x_{i} \}_{i \in V}$ is \cite{choe2004homogeneous} :

\[ M_{G}(x;\lambda) = \sum_{\text{matchings  M}} \prod_{e=ij \in M} \lambda_{e}x_{i}x_{j}  \]

The complementary matching polynomial is defined as 

\[  \Tilde{M}_{G}(x;\lambda) = x^{V} M_{G}(1/x;\lambda)  \]

These matching polynomials also enjoy recursive relations

\[ M_{G}(x;\lambda) = M_{G \setminus e}(x;\lambda) + \lambda_{e} x_{i}x_{j}M_{G-i-j}(x; \lambda) \]

\[ \Tilde{M}_{G}(x;\lambda) = x_{i}\Tilde{M}_{G - i}(x;\lambda) + \sum_{\substack{e \sim i \\ e=ij}} \lambda_{e}\Tilde{M}_{G - i - j}(x;\lambda) \]

\begin{theorem}[Heilmann-Leib Theorem]
Let $G = (V ,E)$ be a loopless graph, and let $\{\lambda_{e}\}$ $e \in E$ be nonnegative edge weights. If $\text{Re} \> x_{i} > 0$ for all $i \in V$ , then

\begin{enumerate}
    \item $ \Tilde{M}_{G}(x;\lambda) \neq 0 $.
    \item For every $i \in V$, $ \Tilde{M}_{G-i}(x;\lambda) \neq 0 $.
    \item For every $i \in V$, $\text{Re} \> \Tilde{M}_{G}(x;\lambda) / \Tilde{M}_{G-i}(x;\lambda) > 0$.
    \item $M_{G}(x;\lambda) \neq 0$. 
\end{enumerate}
Essentially, if the edge weights for the graph $G$ are nonnegative, then the polynomials $M_{G}$ and $\Tilde{M}_{G}$ have the half-plane property.
\end{theorem}

As discussed earlier, transversal matroids enjoy an intricate definition involving bipartite graphs, which is discussed in \cite{choe2004homogeneous} using the Heilmann-Leib Theorem, which we recall here. Consider a bipartite graph $G = (V , E)$ with bipartition $V = A \cup B$, setting $x_{j} = 1$ for $j \in B$ and considering $M_{G}(x; \lambda)$ as a polynomial in $\{x_{i}\}$ $i \in A$ . Then the restricted matching polynomial can be stated as 

\[   \overline{M}_{G}(x;\lambda) = \sum_{\text{matchings} \> M}  \> \prod_{\substack{e=ij \in M \\ i \in A \\ j \in B}} \lambda_{e} x_{i} \]

and it also has the half-plane property, provided that the edge weights are nonnegative. Consider now the transversal matroid $\M[G, A]$ with ground set $A$ defined by the bipartite graph $G$, in which a subset $S \subseteq A$ is declared independent if it can be matched
into $B$. Defining the weighted sum of such matchings,

\[ c(S; \lambda) = \sum_{\substack{\text{matchings \> M} \\ V(M) \cap A = S}} \> \prod_{e \in M} \lambda_{e} \]

Hence,

\[   \overline{M}_{G}(x;\lambda) = \sum_{S \in \mathcal{I}(\M[G, A])}  c(S; \lambda) x^{S} \]

So we see that the matching polynomial, which is also a stable polynomial, is almost the basis generating polynomial of $\M[G, A]$ except the coefficients $c(S; \lambda)$, and this motivated the definition of \emph{nice} transversal matroids. We realize that if all the coefficients have the same value, then the stability of the matching polynomials also applies stability of the basis-generating polynomial. The pair $(G, A)$ is called \emph{nice} if there exists a collection $\{\lambda_{e}\}, e \in E$ of nonnegative edge weights so that $c(S; \lambda)$ has the same nonzero value for all bases $S$ of $\M[G, A]$.

\begin{definition}
The transversal matroid $\M$ is called \emph{nice} if there exists a nice pair $(G, A)$ such that $\M \simeq \M[G, A]$.   
\end{definition}

\section{Lattice path matroids satisfy the half-plane property}

We first recall the iterative description of lattice path matroids proven in \cite{bonin2006lattice},

\begin{theorem}[Theorem 6.7 \cite{bonin2006lattice}]\label{thm:recurs_LPM}
A matroid $\M$ is a lattice path matroid if and only if the ground set can be written as $\{x_{1}, x_{2}, \hdots, x_{k} \}$ so that each restriction $\M_{i}:= \M|\{x_{1}, x_{2}, \hdots, x_{i} \}$ is formed from $\M_{i-1}$ by either
\begin{enumerate}
    \item adding $x_{i}$ as a coloop,
    \item adding $x_{i}$ as a loop, or
    \item adding $x_{i}$ via the principal extension of $\M_{i-1}$ generated by the closure of an independent set of the form $\{x_{h}, x_{h+1}, \hdots, x_{i-1} \}$ for some $h$ with $h < i$. 
\end{enumerate}  
\end{theorem}

We also recall that the half-plane property for a matroid $\M$ is conserved under taking minors, duals, 2-sums, principal extensions, principal truncations, and direct sums \cite[Section 4]{choe2004homogeneous}.

We also would like to mention some important examples of lattice path matroids, that we know satisfy the half-plane property. Firstly, uniform matroids, which are known to be lattice path matroids, do satisfy the half-plane property \cite[Section 9]{choe2004homogeneous}. Additionally, lattice path matroids that are snakes are known to be graphical matroids \cite[Theorem 2.2]{knauer2018tutte} which means that snakes are regular matroids, which in turn implies that basis-generating polynomials for snakes are determinantal, hence snakes are SOS- Rayleigh, which means they are strongly Rayleigh and therefore snakes do satisfy the half-plane property.

With these results, we now can present the following result on lattice path matroids,

\begin{theorem}\label{thm:LPM_not_HPP}
Lattice path matroids satisfy the half-plane property. \end{theorem}

\begin{proof}
We consider $\M$ to be a rank $k$ lattice path matroid on $n$ elements. Firstly, we consider the case when $\M$ is connected. We proceed in our proof via induction on the number of elements $n$ of $\M$. For the base case, we consider $n=2$, we realize that $\M$ is either a uniform matroid or a snake, and in both these cases we realize that the basis generating polynomial does satisfy the half-plane property as proven previously in \cite{choe2004homogeneous, knauer2018tutte}. 

Hence, we proceed with the case where we consider the case that all lattice path matroids on $n$ elements satisfy the half-plane property. We now consider a lattice path matroid $M$ on $n+1$ elements and by Theorem \ref{thm:recurs_LPM} we know that $M$ is obtained in a recursive way by the addition of a loop, coloop, or via principal truncation to a matroid $M'$ on $n$ elements. In the case of the addition of a loop, we know that the basis generating polynomial of $M$ and $M'$ is the same, and  $M'$ does satisfy the half-plane property because of our induction hypothesis, therefore $M$ also satisfies the half-plane property. Similarly, if we consider the addition of a coloop $e$ and if $f$ is the basis generating polynomial of $\M'$, then $f$ satisfies the half-plane property and the basis generating polynomial of $\M$ is $g = x_{e} \cdot f$. We compute the \emph{Rayleigh difference} $\Delta_{i,j}$ for the polynomial $g$ for $\{i,j\} \subseteq \{[n] \cup e \}$ $(i \neq e, j \neq e)$,

\[  \frac{ \partial g} {\partial x_{i}} = x_{e} \cdot \frac{ \partial f} {\partial x_{i}} , \quad \frac{ \partial g} {\partial x_{j}} = x_{e} \cdot \frac{ \partial f} {\partial x_{j}}, \quad \frac{ \partial^{2} g} {\partial x_{i} \partial x_{j}} = x_{e} \cdot \frac{ \partial^{2} f} {\partial x_{i} \partial x_{j}} \]

\[  \Delta_{ij}(g)  = \frac{ \partial g} {\partial x_{i}} \frac{ \partial g} {\partial x_{j}} - \frac{ \partial^{2} g} {\partial x_{i} \partial x_{j}} \cdot g = (x_{e})^{2} (\Delta_{ij}(f)) \geq 0 \]

and see that it is non-negative in this case. For the other possible case we consider $i=e$ and $j \in [n]$, the Rayleigh difference is given as

\[  \frac{ \partial g} {\partial x_{e}} = f , \quad \frac{ \partial g} {\partial x_{j}} = x_{e} \cdot \frac{ \partial f} {\partial x_{j}}, \quad \frac{ \partial^{2} g} {\partial x_{i} \partial x_{j}} = \frac{ \partial f} {\partial x_{j}} \]

\[  \Delta_{ij}(g)  = \frac{ \partial g} {\partial x_{i}} \frac{ \partial g} {\partial x_{j}} - \frac{ \partial^{2} g} {\partial x_{i} \partial x_{j}} \cdot g = ( x_{e} \cdot f \cdot \frac{ \partial f} {\partial x_{j}} -  x_{e} \cdot f \cdot \frac{ \partial f} {\partial x_{j}}) = 0 \]

and is also non-negative in this case, which means that $M$ also satisfies the half-plane property. 

The remaining case is when the matroid $\M$ is obtained via a principal extension of a matroid on $n$ elements and we know that the half-plane property is conserved for a \emph{nice} principal extension \cite[Proposition 4.12]{choe2004homogeneous}. Consider the set $F = \{x_{1}, x_{2}, \hdots, x_{i-1} \}$ and we consider principal extension of $\M$ by the addition of an element $x_{i}$ to $\M=\M_{i-1}$ via $F$. We know that the bases of $\text{tr}_{F}(\M)$ are of the form $B \setminus f$ where $f \in F$ and the bases of the principal extension $\M +_{F} x_{i}$ are of the form 

\[ \mathcal{B}(\M +_{F} x_{i}) = \mathcal{B}(\M) \cup \{ B \cup \{x_{i}\} : B \in \mathcal{B}(\text{tr}_{F}(\M)) \}   \]

To ensure that this principal extension is nice, the coefficients $\lambda_{f}, f \in F$ need to satisfy the equality 

\[ \sum_{f \in F: B \cup \{f\} \in \mathcal{B}(\M)} \lambda_{f} = 1  \]

where $B \in \mathcal{B}(\text{tr}_{F}(\M))$. We see that for our given choice of $F$, let $S = \{f \in F:   B \cup \{f\} \in \mathcal{B}(\M), B \in \mathcal{B}(\text{tr}_{F}(\M))\} \subseteq F$ and let $|S| = k \leq |F|$. Then if fix the value of 

\[ \lambda_{j} = \frac{1}{k} , j \in S \]

then we see that the equation for a nice principal extension is satisfied and therefore by \cite[Proposition 4.12]{choe2004homogeneous} $\M$ also satisfies the half-plane property in this case. With this, we prove that all connected lattice path matroids satisfy the half-plane property.

We now consider the case when the lattice path matroid is disconnected. In this case $\M$ can be expressed as a direct sum of $\M_{1}, \hdots, \M_{k}$, where each $\M_{i}$ is a connected lattice path matroid which we showed satisfy the half-plane property. But we know that the half-plane property is also conserved under taking direct sums \cite[Section 4.3]{choe2004homogeneous}, therefore $\M$ satisfies the half-plane property in this case as well. Hence, the proof.

\end{proof}    

\begin{example}\label{eg:Example_principal}
We illustrate with the help of an explicit example of a lattice path matroid $\M$, the \emph{nice} principal extension obtained via the addition of an element. Consider the lattice path matroid $\M = M[125,356]$ of rank three on six elements and we consider $F = \{ x_{1}, x_{2}, \hdots, x_{6} \}$ and we consider the principal extension of $\M$ by the element $x_{7}$ via $F$. $\M$, $\text{tr}_{F}(\M)$ and the principal extension $\M +_{F} \{x_{7}\}$ are depicted in Figure \ref{fig:eg_nice_princ} and we realize that   

\begin{figure}
    \centering
\includegraphics[scale=0.4]{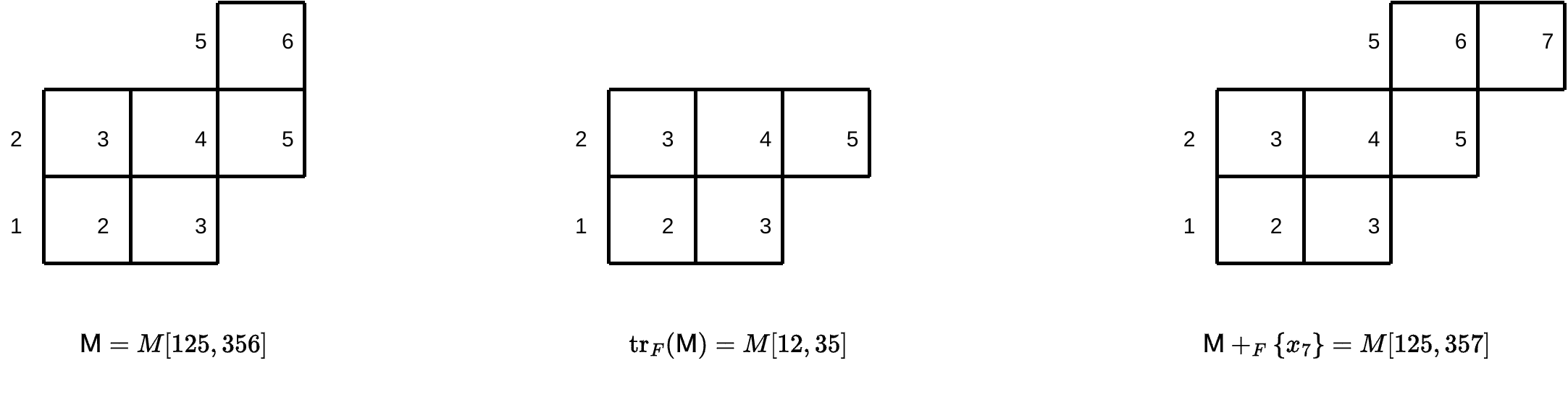}
    \caption{The matroid $\M$ discussed in Example \ref{eg:Example_principal} along with its truncation and principal extension.}
    \label{fig:eg_nice_princ}
\end{figure}

\[ \mathcal{B}(\M) = \{125,126,135,136,145,146,156,235,236,245,246,256,345,346,356 \}  \]

\[ \mathcal{B}(\text{tr}_{F}(\M)) = \{ 12,13,14,15,23,24,25,34,35 \} \]

\begin{multline*}
\mathcal{B}(\M +_{F} \> x_{7}) = \{ 125,126,135,136,145,146,156,235,236,245,246,256,345, \\ 346,356,127,137,147,157,237,247,257,347,357\}    \end{multline*}

We realize that the set $S = \{f \in F:   B \cup \{f\} \in \mathcal{B}(\M), B \in \mathcal{B}(\text{tr}_{F}(\M))\} = \{1,2,3,4,5,6\} = F$ and we can fix the value of coefficient $\lambda_{j} = 1/6, j \in S$, such that the principal extension is \emph{nice}.

\end{example}

\section{Bicircular Matroids and Positroids}

In this section, we consider a bicircular matroid $\M = \M(E, \mathcal{B}(\M))$ with a linear order $\leq$ on the ground set $E$, and we are interested to know whether $\leq$ defines a \emph{positroid order} on $\M$ in the sense of Bonin  \cite{bonin2023characterization}. An affirmative conclusion to this would mean that bicircular matroids are also positroids.

We point out that for a subclass of matroids to be positroids is not dependent necessarily on any one specific property. For contrast, if we consider graphical (also referred to as circular in \cite{kassel2022determinantal}) matroids and bicircular matroids, then we recall that these two are the only classes of matroids on the set of edges of a graph for which the set of circuits consists in all subgraphs homeomorphic to a given family of connected graphs \cite{pereira1972subgraphs}. However, we know that not all graphical matroids are positroids, notably if we consider the graphical matroid on the complete graph on four vertices $K_{4}$, i.e., $\M(K_{4})$, then we know that it is not a positroid \cite{bonin2023characterization}.
We also explain another motivation for this question. We show that the class of bicircular matroids already contains a class of matroids that are known to be positroids: namely the class of \emph{lattice path bicircular matroids}, which as the name suggests are matroids which are both lattice path matroids and bicircular matroids and were introduced and studied in \cite{guzman2022lattice}. We know by the work of Oh in \cite{oh2011positroids} that lattice path matroids are positroids, therefore lattice path bicircular matroids are also positroids. The authors in \cite{guzman2022lattice} also provide a list of excluded minors for bicircular matroids, which are not lattice path matroids, via the following result,

\begin{theorem}[Theorem 2 \cite{guzman2022lattice}]\label{thm:Bicirc_not_LPM}
A bicircular matroid is a lattice path matroid if and only if it has no one of the following matroids as a minor:

\[ C^{2,4}, \mathcal{W}^{3} , A^{3}, R^{3}, R^{4}, D^{4}, B^{1}, \> \text{and} \> S^{1} . \]   
\end{theorem}

We refer the reader to \cite[Figure 2]{guzman2022lattice} for the affine and bicircular representations of the six excluded minors listed in Theorem \ref{thm:Bicirc_not_LPM}. 

We now discuss the class of \emph{base sortable matroids} first introduced by Blum in \cite{blum2001base}. By the work of Bonin \cite{bonin2023characterization} we know that \emph{base sortable matroids} are positroids, 

\begin{theorem}[Theorem 2.5 \cite{bonin2023characterization}]
Let $M$ be a matroid. A linear order on $E(M)$ is a positroid order if and only if it is a base-sorting order.    
\end{theorem}

Moreover, Blum listed all rank three excluded minors for the class of base-sortable matroids in \cite{blum2001base}, which we use subsequently for our results

\begin{corollary}[Corollary 4.12 \cite{blum2001base}]\label{cor:excluded-minor_pos}
Let $M$ be a matroid on $[d]$ of rank $3$. Then $M$ is base-sortable if and only if $M$ has no deletion $N$ with geometric representation:

\begin{enumerate}
    \item A $k$-gon, $k \geq 3$, whose edges are rank $2$ circuits, and an additional generic point, i.e.,
    \begin{figure}[H]
        \centering
        \includegraphics[scale=0.45]{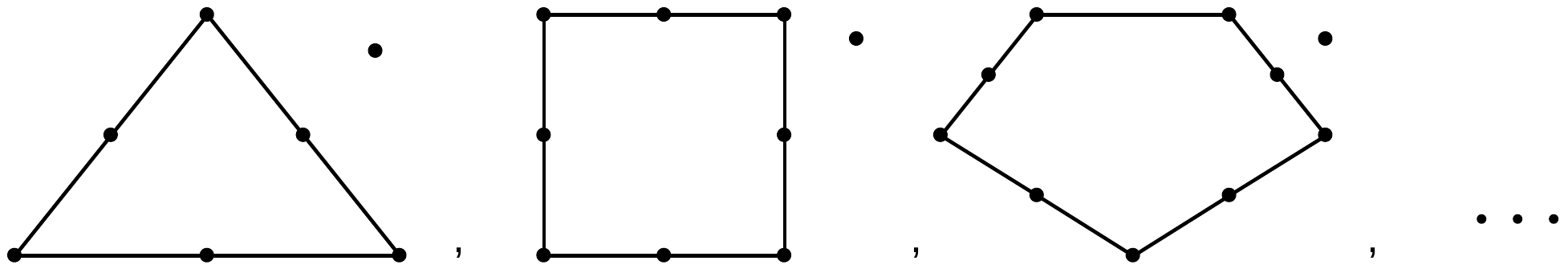}
        \label{fig:exclude_minor_1}
    \end{figure}
    \item One of the following:
        \begin{figure}[H]
        \centering
        \includegraphics[scale=0.425]{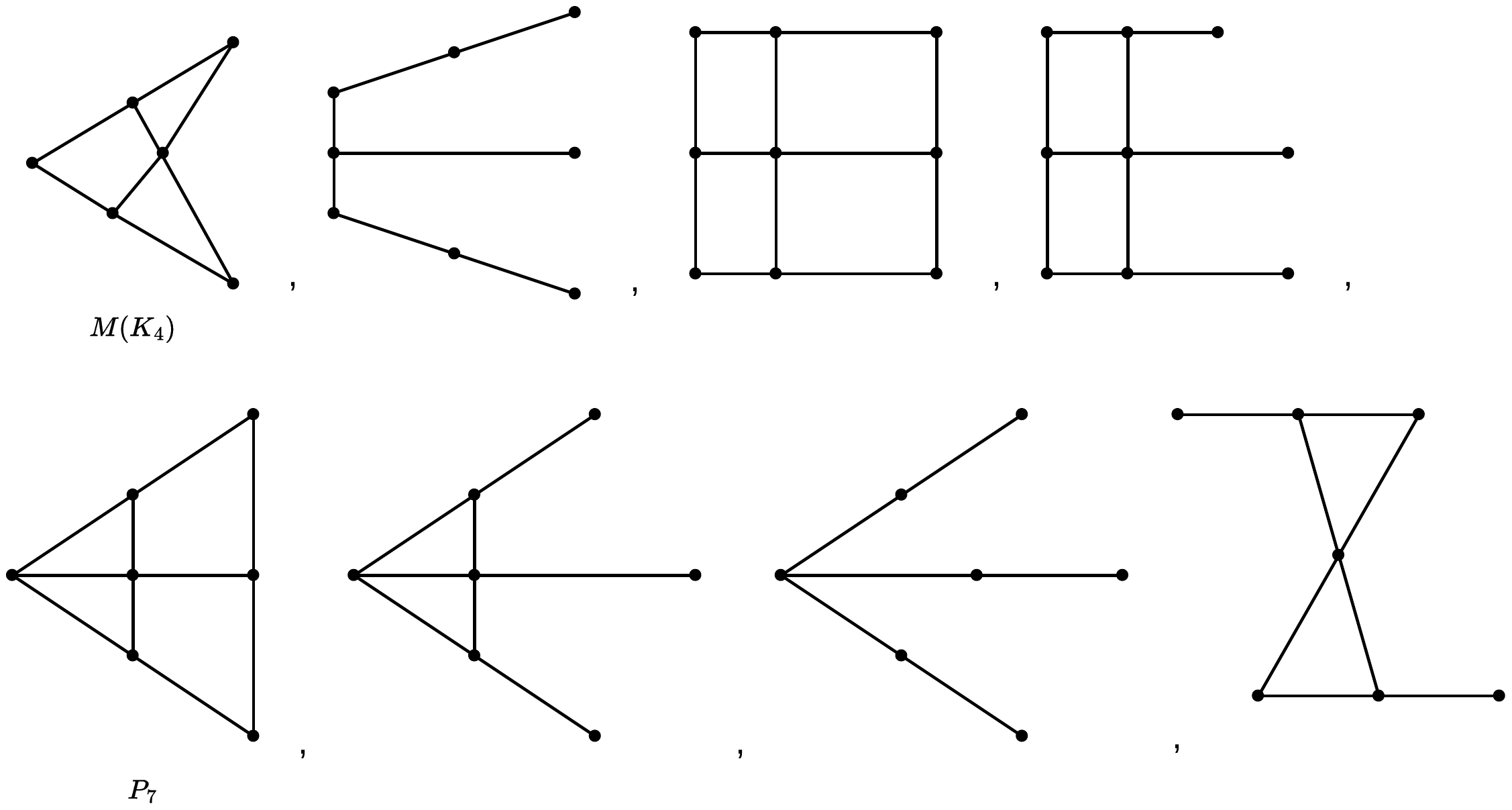}
        \label{fig:exclude_minor_2}
    \end{figure}
    where $M(K_{4})$ is the graphic matroid defined by $K_{4}$, the complete graph on $4$ vertices.
\end{enumerate}
\end{corollary}

Therefore, Corollary \ref{cor:excluded-minor_pos} provides us a list of excluded minors for positroids of rank at most three. 

\begin{figure}
    \centering
    \includegraphics[scale=0.5]{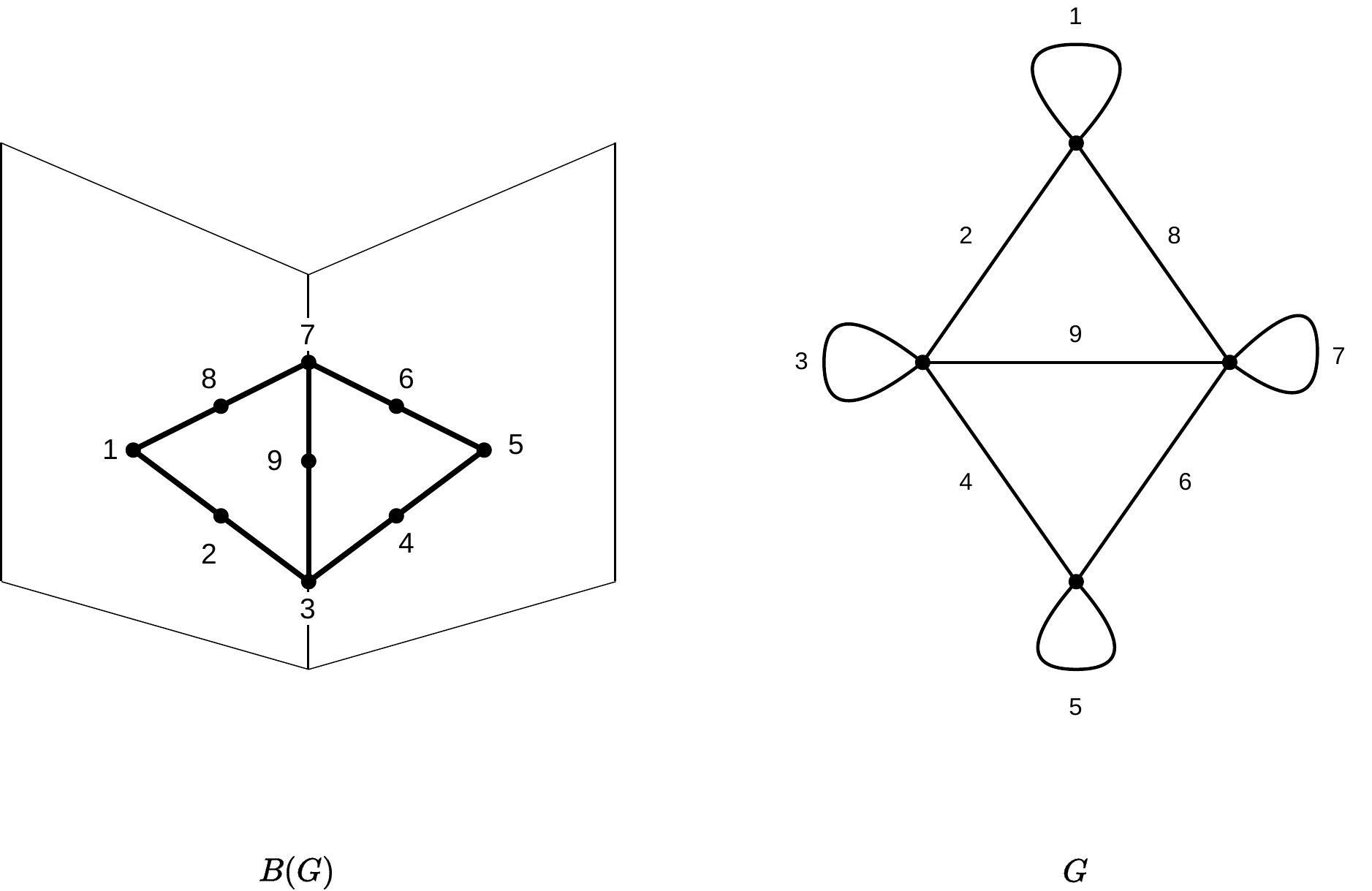}
    \caption{A rank four bicircular matroid $M$ which is not a positroid with one of its bicircular representations.}
    \label{fig:excl_min_del}
\end{figure}

\begin{example}\label{eg:bicicrc_not_pos}
Consider the bicircular matroid $B(G)$ illustrated in Figure \ref{fig:excl_min_del} along with its bicircular representation, also discussed in the context of excluded minors of positroids in \cite{bonin2023characterization}. In this case, $\M = B(G)$ is an excluded minor of positroids of rank four and hence is an example of a bicircular matroid that is not a positroid.
\end{example}

\begin{corollary}\label{cor:bicirc_and_pos}
The classes of bicircular matroids and positroids are incomparable, i.e., neither of these two classes is contained in the other.    
\end{corollary}

\begin{proof}
It is clear by Example \ref{eg:bicicrc_not_pos} that the class of bicircular matroids does not lie inside the class of positroids. For example to show the non-containment of positroids inside bicircular matroids, one can consider any lattice path matroid, which is not bicircular, and such matroids have been listed in \cite[Figure 7]{hogan2023excluded}.   
\end{proof}

\begin{remark}
We do want to highlight some recent work done on the negative correlation property of basis elements of bicircular matroids. In \cite{kassel2022determinantal} it is shown that there exists a non-uniform determinantal probability measure on the bases of bicircular matroids, building on previous work done on studying probability measures on cycle-rooted spanning forests in \cite{kenyon2011spanning, lyons2003determinantal} which also turns out to be strongly Rayleigh based on results from \cite[Proposition 3.5]{borcea2009negative}. However, since this measure is non-uniform, its generating polynomial is not the same as the basis generating polynomial of the matroid, hence these results can only be extended to show that bicircular matroids satisfy the \emph{weak half-plane property}, which although is already known to be true for the superclass of transversal matroids \cite[Corollary 8.2]{choe2004homogeneous}.    
\end{remark}

\section{Sparse paving matroids are not Rayleigh}

We recall the following definitions,

\begin{definition}
A matroid $\M$ is said to be \emph{paving} if it has no circuits of size less than the rank $r(\M)$.   
\end{definition}

\begin{definition}
A matroid $\M$ is said to be \emph{sparse paving} if $\M$ and its dual $\M^{*}$ both are paving.    
\end{definition}

Paving and sparse paving matroids are immensely important classes of matroids, especially in understanding the asymptotics of matroids, as it is conjectured that almost all matroids are paving. With the enumeration of all matroids on up to 8 elements done in \cite{blackburn1973catalogue}, it was speculated by Crapo and Rota that paving matroids would dominate any enumeration of matroids, which is stated as a conjecture in \cite[Conjecture 1.6]{mayhew2011asymptotic} and they also conjecture that asymptotically almost every matroid has a sparse paving matroid as a minor \cite[Conjecture 1.7]{mayhew2011asymptotic}. Then in \cite{pendavingh2015number} Pendavingh and Pol prove that on log scale, almost all matroids are sparse paving, i.e., if $m_{n}$ denotes the number of matroids on $n$ elements and $s_{n}$ denotes the number of sparse paving matroids on $n$ elements, then 

\[ \lim_{n\to\infty} \frac{\log s_{n}}{\log m_{n}} = 1  \]

and it is conjectured that if the log factor is removed, then almost all matroid are sparse paving.

In \cite{jerrum2006two}, Jerrum proves the first result concerning negative correlation properties of sparse paving matroids.

\begin{lemma}[Lemma 2 \cite{jerrum2006two}]\label{lem:Lem_jerrum}
Sparse paving matroids are balanced.    
\end{lemma}

At this juncture, we also would like to recall some results concerning 3-connected matroids. Firstly, in  \cite{LOWRANCE2013115} Lowrance, Oxley, Semple and Welsh proved the following result about the asymptotics of 3-connected matroids,

\begin{theorem}[Theorem 4.2\cite{LOWRANCE2013115}]\label{thm:3-connec_almost}
Almost all n-element matroids are 3-connected.    
\end{theorem}

Also, 3-connectivity is intrinsically linked with 2-sums,

\begin{theorem}[Theorem 8.3.1 \cite{oxley}]
A $2-$connected matroid $\M$ is not $3-$connected if and only if $\M = \M_{1} \oplus_{2} \M_{2}$ for some matroids $\M_{1}$ and $\M_{2}$, each of which has at least three elements and is isomorphic to a proper minor of $\M$.
\end{theorem}

We now prove the following result about sparse paving matroids,

\begin{theorem}\label{thm:sparse_not_Rayleigh}
The class of sparse paving matroids is not Rayleigh.
\end{theorem}

\begin{proof}
We proceed by contradiction and base our arguments on the dominating classes in the enumeration of matroids. We assume that sparse paving matroids are Rayleigh, and let $m_{n}$ denote the number of matroids on $n$ elements, $r_{n}$ denote the number of Rayleigh matroids on $n$ elements, and $s_{n}$ denote the number of sparse paving matroids on n elements. Then by our assumption $s_{n} < r_{n}$. Also, since the class of Rayleigh matroids is closed under $2$-sums, the number of Rayleigh matroids that are obtained as a $2-$sum, which we denote by $r^{\text{2-sum}}_{n}$, provides a dominating class of matroids on log scale because by the work in \cite{pendavingh2015number} we know that sparse paving matroids dominate any enumeration of matroids on log-scale. But by Theorem \ref{thm:3-connec_almost} we know that even on log scale, almost all matroids are 3-connected, which would imply that they would not be obtained via a 2-sum. This gives us a contradiction. Hence, the proof.
\end{proof}

\begin{remark}
We do want to point the reader to one exercise sheet of a Nordfjordeid Summer School 2018 by June Huh \cite{Huh2018}, where June Huh comments about the Master's thesis of Alejandro Erickson \cite[Theorem 4.2.1]{alejandro} wherein it is proven that sparse paving matroids are Rayleigh. The proof involves a stronger assertion that the coefficients in the Rayleigh difference of a sparse paving matroid are also positive. However, it is mentioned in \cite[Problem 29]{Huh2018} that Benjamin Schroeter proposed a counter-example to this assertion in the form of the graphic matroid on the complete graph $K_{4}$ and our proof provides a solution to Problem 30 in \cite{Huh2018}.  
\end{remark}

\section{Conclusions and Future Work}

Our results in this article help us understand the half-plane property, especially when considered over the class of transversal matroids \cite[Question 13.17]{choe2004homogeneous}. Some previously known classes of matroids that do satisfy the half-plane property are regular matroids, uniform matroids, sixth roots of unity matroids, and all matroids of rank two. Recently, in \cite[Theorem 5.2]{kummer2023}, the authors conducted a census on all matroids on at most eight elements to list matroids that do not satisfy the half-plane property and they provide a list of  22 sparse paving matroids of rank four on eight elements that do not satisfy the half-plane property. The authors use a criterion of checking the half-plane property stated in \cite[Theorem 3]{wagner2009criterion} which states that positive Rayleigh difference for a single pair of elements $(i,j)$ is sufficient enough to check for half-plane property and the authors in \cite{kummer2023} develop an algorithm based on this to verify half-plane property of matroids. This is a complementary approach to our results and can also be viable for checking the half-plane property for other classes of matroids.

Another aspect of our work is to understand the $c$-Rayleigh property for the basis generating polynomials of the class of matroids, especially for the class of transversal matroids. A polynomial $f \in \mathcal{P}_{n}$ is called $c$-Rayleigh if 

\[ \frac{ \partial f} {\partial z_{i}} (x) \frac{ \partial f} {\partial z_{j}} (x) \geq c \frac{ \partial^{2} f} {\partial z_{i} \partial z_{j}} (x) f(x)  \] 

for all $x = (x_{1}, \hdots , x_{n}) \in \mathbb{R}^{n}_{+}$ and $1 \leq i, j \leq n$, where $\mathbb{R}_{+} = (0, \infty)$.

\begin{figure}
    \centering
    \includegraphics[scale=0.45]{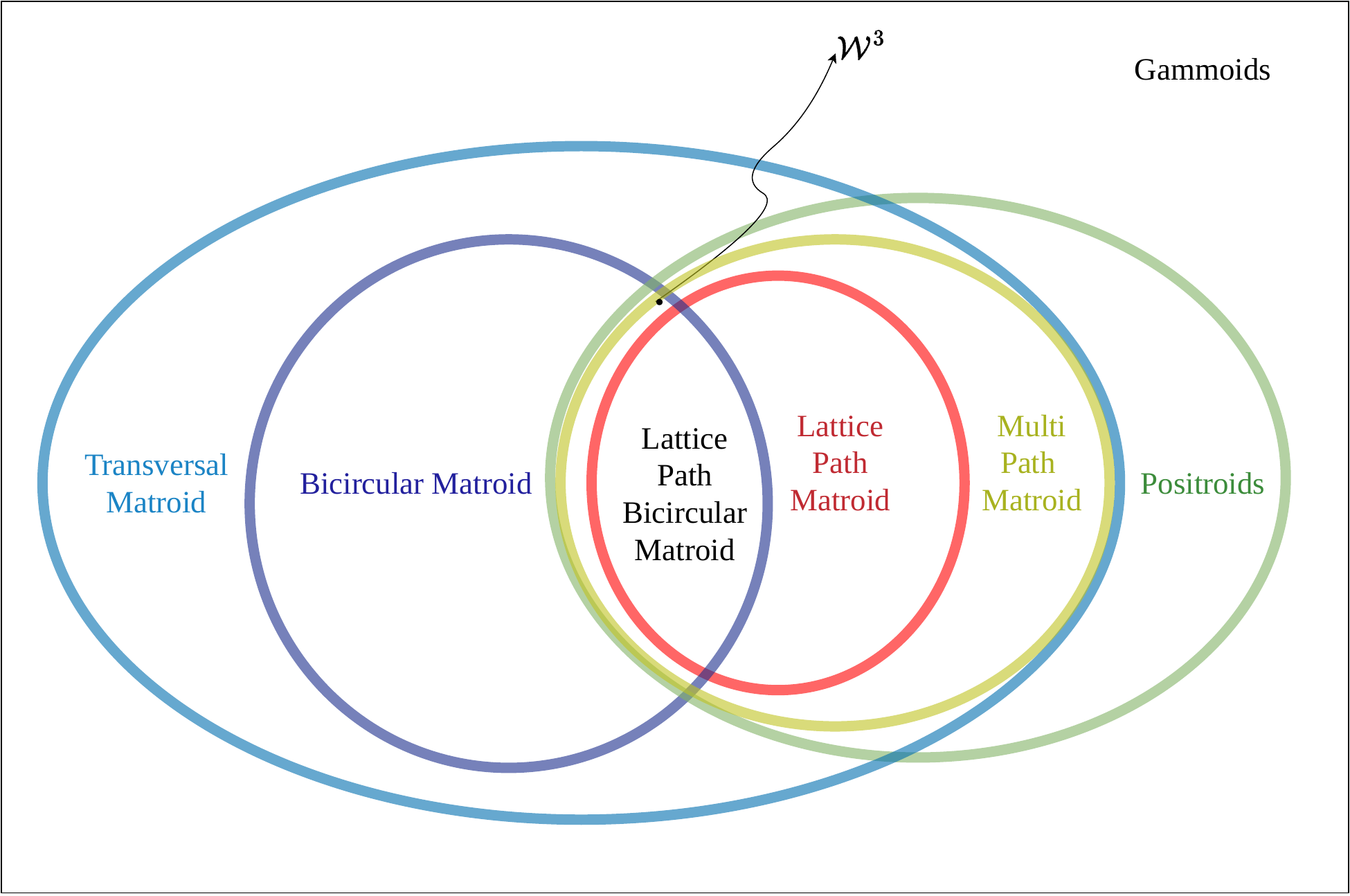}
    \caption{A Venn diagram depicting how the classes of transversal matroids, positroids, bicircular matroids, lattice path matroids, multi path matroids, lattice path bicircular matroids intersect in the superclass of gammoids.}
    \label{fig:venn_diag_matroid}
\end{figure}

In the seminal work of Branden and Huh \cite{branden2020lorentzian} the authors proposed the following conjecture concerning $c-$ Rayleigh polynomials,

\begin{conjecture}[Conjecture 3.12 \cite{branden2020lorentzian}]\label{conj:lorent_conj}
The following conditions are equivalent for any non-empty $J \subseteq \{0, 1\}^{n}$ :

\begin{enumerate}
    \item $J$ is the set of bases of a matroid on $[n]$.
    \item The generating function $f_{J}$ is a homogeneous $\frac{8}{7}$ -Rayleigh polynomial.
\end{enumerate}
\end{conjecture}

The problem to verify this conjecture on transversal matroids was suggested in \cite[Problem 30]{tewari2023positroids}. Our results here provide new perspectives in the context of this problem. Theorem \ref{thm:sparse_not_Rayleigh} suggests that the class of Rayleigh matroids might not necessarily dominate an enumeration of all matroids although the framing of Conjecture \ref{conj:lorent_conj} and the fact that only a few explicit examples of $c-$Rayleigh matroids are known for $c > 1$ may suggest otherwise. This encourages the approach to look out for classes of matroids which are $c-$Rayleigh,  $c > 1$.

We know that the class of transversal matroids is not minor closed, although the class of Rayleigh matroids is minor-closed \cite{choe2006rayleigh}, and therefore one problem to pursue in the future could be to verify whether almost all transversal matroids are Rayleigh i.e., $1-$Rayleigh, whereas till now there are only few known explicit examples of transversal matroids which are $c-$Rayleigh, where $1 < c \leq \frac{8}{7}$, for instance, \cite[Proposition 5.9]{choe2006rayleigh}.

Figure \ref{fig:venn_diag_matroid} illustrates various classes of matroids namely transversal matroids, positroids, bicircular matroids, lattice path matroids, multi-path matroids, lattice path bicircular matroids and how they intersect in the superclass of gammoids by a Venn diagram. We point out that positroids, lattice path matroids, multi-path matroids, and lattice path bicircular matroids are all Rayleigh. A problem to possibly work in this context can be the following,

\begin{problem}\label{prob:problem1}
Verify whether the class of bicircular matroids is Rayleigh. Do they also satisfy the half-plane property?    
\end{problem}

With previously known results concerning matroids that satisfy the half-plane property, for a possible counterexample to Problem \ref{prob:problem1} one could begin with finding those bicircular matroids that are not graphic and lattice path matroids since such bicircular matroids would satisfy the half-plane property. We take this opportunity to show that although the half-plane property is conserved under most matroidal operations, finding explicit examples of matroids that do or do not satisfy the half-plane property can be tricky. 

\begin{example}
We consider the example of the matroid $\M$ of rank three on nine elements described in Figure \ref{fig:tic-tac-toe} which is obtained via relaxations of the non-Pappus matroid, and is also referred to as the \emph{tic-tac-toe matroid} in \cite{chun2016bicircular} (the usual \emph{tic-tac-toe} matroid is defined on a $[3] \times [3]$ grid). We know that the non-Pappus matroid does not satisfy the half-plane property \cite{kummer2023}. But $\M = \B^{*}(K_{3,3})$, i.e., $\M$ is the dual matroid to the bicircular matroid on the complete bipartite graph $K_{3,3}$ and by Matthews characterization of graphs for which bicircular matoroids are graphic \cite{matthews}, we conclude that $\B(K_{3,3}) = M(K_{3,3})$, where $M(K_{3,3})$ represents the graphical matroid on $K_{3,3}$. Therefore, $\B(K_{3,3})$ satisfies the half-plane property and since the half-plane property is conserved under duality \cite{choe2004homogeneous}, $\M = \B^{*}(K_{3,3})$ also satisfies the half-plane property. Hence, using relaxations, we obtain a matroid that satisfies the half-plane property from a matroid that does not satisfy the half-plane property. In the other direction, it is noted in \cite{choe2004homogeneous}, the matroid $F^{-3}_{7}$ does not satisfy the half-plane property although it is a relaxation of the matroid $P_{7}$, which is a sixth root of unity matroid and hence satisfies the half-plane property. 

\begin{figure}
    \centering
    \includegraphics[scale=0.48]{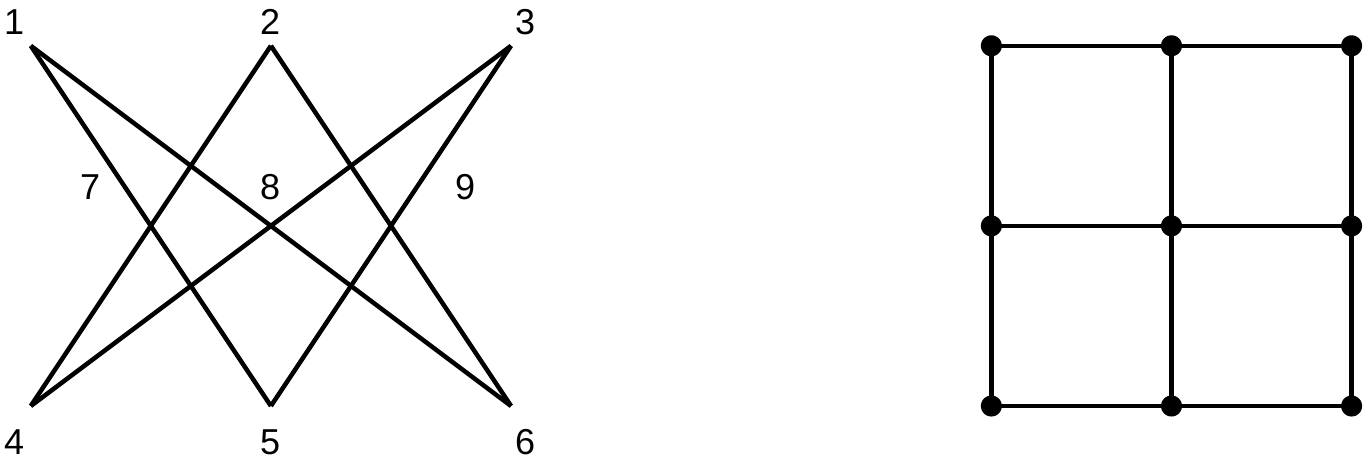}
    \caption{The tic-tac-toe matroid on a $[2] \times [2]$ grid which is obtained as relaxation of \emph{non-Pappus} matroid.}
    \label{fig:tic-tac-toe}
\end{figure}
\end{example}

Another conjecture first posed in \cite{choe2004homogeneous} which is still open and intricately related to our results is the following 

\begin{conjecture}[Conjecture 13.16 \cite{choe2004homogeneous}]\label{conj:trans_three}
All rank-3 transversal matroids have the half-plane property.    
\end{conjecture}

In \cite[Section 10.4]{choe2004homogeneous}, the authors provide a characterization of rank three transversal matroids in 8 separate classes. Out of these 8 classes, one is of uniform matroids which we know satisfy the half-plane property. Another class is $L_{n_{1}:n'}, n_{1} \geq 3, n' \geq 1$, consisting of one $n_{1}$-point line together with $n'$ freely added points. On comparing this family of transversal matroids with the list of excluded minors for the class of lattice path matroids \cite[Theorem 3.1]{bonin2010lattice}, we notice that this family of matroids is also a subclass of lattice path matroids and hence they satisfy the half-plane property. The remaining six classes of rank three transversal matroids still need to be checked whether they satisfy the half-plane property and we wish to pursue the complete verification of rank three transversal matroids in future work. 

Since we know that lattice path matroids are also positroids, we can also ask a similar question as Conjecture \ref{conj:trans_three} for the class of positroids as also stated in \cite[Problem 32]{tewari2023positroids}

\begin{problem}\label{prob:problem2}
Verify whether all rank three positroids satisfy the half-plane property. Additionally, do all positroids satisfy the half-plane property?
\end{problem}

\bibliographystyle{siam}

\bibliography{biblio.bib}

\end{document}